\documentclass[twoside,12pt,letter]{article}
\usepackage{bbm}
\usepackage{mathrsfs}
\usepackage{amsfonts}
\usepackage{amsmath}
\usepackage{amssymb,amsthm}
\newtheorem{theorem}{Theorem}[section]%
\newtheorem{lemma}[theorem]{Lemma}%
\newtheorem{cor}{Corollary}[section]%
\newtheorem{prop}[theorem]{Proposition}%

\newtheorem{exam}{Example}[section]

\newcommand{\bref}[5]{\bibitem{#1} {#2} {\it #3} {\bf #4}#5.}

\begin{document}
  \title{The critical group of $C_4\times C_n$
 \thanks{Supported by NSF of the People's Republic of China(Grant
 No. 10871189 and  No. 10671191).}
  }
 \author{Jian Wang,  Yong-Liang Pan\thanks{Corresponding author. Email: ylpan@ustc.edu.cn},\,\,
   \\
  {\small Department of Mathematics, University of Science and Technology
         of China}\\
  {\small Hefei, Auhui 230026, The People's Republic of China}\\
}
\date{}
\maketitle {\centerline{\bf\sc Abstract}\vskip 8pt In this paper,
the  critical group structure of the Cartesian product graph $C_4\times C_n$ is
determined, where $n\ge 3$.



\section{Introduction}

Let $G=(V,E)$ be a finite connected graph without self-loops, but
with multiple edges allowed. Then the Laplacian matrix of $G$ is the
$|V|\times |V|$ matrix defined by
 $$L(G)_{uv}=\left\{
\begin{array}{ll}
d(u), &  \text{if}\hspace{0.2cm}u=v, \\
-a_{uv}, & \text{if}\hspace{0.2cm}u\not=v,
\end{array}
\right.\eqno(1.1) $$ where $a_{uv}$ is the number of the edges
joining $u$ and $v$, and $d(u)$ is the degree of $u$.

Regarding $L(G)$ as representing an abelian group homomorphism:
$Z^{|V|}\rightarrow Z^{|V|}$, its cokernel
coker$(L(G))={Z}^{|V|}/\text{im}\, (L(G))$ is an abelian group,
determined by the generators $x_1,\cdots, x_{|V|}$ and  relation
$L(G)X=0$, where $x_i=(0,\cdots,0,1,0,\cdots,0)\in {Z}^{|V|}$, whose
unique nonzero 1 is in position $i$, and $X=(x_1,\cdots,x_{|V|})^t$.
Note that the same symbol $x_i$ denotes both an element of the group
coker$(L(G))$ and a basis element of the free abelian group
${Z}^{|V|}$.

The  finitely generated abelian group coker$(L(G))$ can be described
in terms of  the Smith normal form (or simply SNF) of $L(G)$. Two
integral matrices $A$ and $B$ of order $n$ are equivalent (written
by $A \sim B$) if there are unimodular  matrices $P$ and $Q$ such
that $B=PAQ$. Equivalently, $B$ is obtainable from $A$ by a sequence
of elementary row and column operations: (1) the interchange of two
rows or columns, (2) the multiplication of  any row or column by
$-1$, (3) the addition of any integer times of one row (resp.
column) to another row (resp. column).
 It is easy to see that $A\sim B$  implies that coker$(A)\cong$ coker$(B)$.
Given any  $|V|\times |V|$ unimodular matrices $P$ and $Q$ and any
integral matrix $A$ with $PAQ=$diag$(a_1,\cdots, a_{|V|})$, it is
easy to see that $Z^{|V|}/\mbox{im}(A)\cong (Z/ a_1
Z)\oplus\cdots\oplus(Z/ a_{|V|}Z)$. Here, the rank of $L(G)$ is
$|V|-1$, with kernel generated by the transpose of the vector
$(1,\cdots,1)$. Thus we can assume  the SNF of $L(G)$
is  diag$(t_1,\cdots,t_{|V|-1},0)$, and it induces an isomorphism
$$\mbox{coker}(L(G))\cong K(G)\oplus Z.\eqno(1.2)$$
where $K(G)=\left({Z}/ t_1 {Z}\right)\oplus \left({Z}/t_2 {Z}\right)
\oplus\cdots\oplus\left({Z}/ t_{|V|-1} {Z}\right)$.

In [1] and [5 (Chapter 14)], the finite abelian group $K(G)$ is
defined to be the critical group of $G$. Its invariant factors
$t_1,t_2,\cdots t_{|V|-1}$  can be computed in the following way:
for $1\le i<|V|$, $t_i=\Delta_i/\Delta_{i-1}$ where $\Delta_0=1$ and
$\Delta_i$ is the $i-$th determinantal  divisor of $L(G)$, defined
as the greatest common divisor of all $i\times i$ minor
subdeterminants of $L(G)$. From the well known Kirchhoff's
Matrix-Tree Theorem [7, Theorem 13.2.1] we know that $t_1\cdots
t_{|V|-1}$
 equals the number $\kappa$ of spanning trees of $G$. It follows that  the invariant factors of $K(G)$ can be used to
distinguish pairs of non-isomorphic graphs which have the same
$\kappa$, and so there is considerable interest in their properties.
If $G$ is a simple connected graph, the invariant factor $t_1$ of
$K(G)$ must be equal to 1, however, most of them are not easy to be
determined.

Compared to the number of the results on the spanning tree number
$\kappa$, there are relatively few results describing the critical
group structure of $K(G)$ in terms of the structure of $G$. There
are also very few interesting infinite family of graphs for which
the group structure has been complete determined (see [2,\, 3,\,
4,\, 6,\, 7,\, 8], and the references therein). In this paper, we
describe  the  critical group structure of Cartesian product graph
$C_4\times C_n\, (n\geq 3)$ completely, where $C_n$ is the cycle on
$n$ vertices.

Given two disjoint graphs $G_1=(V_1,E_1)$ and $G_2=(V_2,E_2)$, their
Cartesian product is the graph $G_1\times G_2$ whose vertex set is
the cartesian product $V_1\times V_2$. Suppose $u_1,\, u_2\in V_1$
and $v_1,\, v_2\in V_2$. Then $(u_1,v_1)$ is adjacent to $(u_2,v_2)$
if and only if one of the following conditions satisfied: (i)
$u_1=u_2$ and $(v_1 ,v_2)\in E_2$, or (ii) $(u_1,u_2)\in E_1$ and
$v_1=v_2$. One may view $G_1\times G_2$ as the graph obtained from
$G_2$ by replacing each of its  vertices with a copy of $G_1$, and
each of its edges with $|V_1|$ edges joining corresponding vertices
of $G_1$ in the two copies.  From the definition of the Cartesian
product of two graphs, it is easy to see that there are $n$ layers
of $C_4\times C_n$, each of which is a copy of $C_4$. Let ${Z}_n$
denote ${Z}/n{Z}$, then for $i\in {{Z}}_n,\;  j\in {Z}_4$, let
$v_j^i$ denote the $j$-th vertex in the $i$-th layer of $C_4\times
C_n$.  The vertex $v_j^i$ is adjacent to vertices $v_j^{l}$ and
$v_k^i$, where $l=i\pm 1$,\, $k=j\pm 1(\mod 4)$\,
(see Fig. 1).\\
\setlength{\unitlength}{1.0cm}
\begin{picture}(16,6.5)
\put(1,0){\line(0,1){3}} \put(1,0){\line(1,1){1}}
\put(2,4){\line(-1,-1){1}} \put(2,4){\line(0,-1){3}}
\put(3,0){\line(0,1){3}} \put(3,0){\line(1,1){1}}
\put(4,4){\line(-1,-1){1}} \put(4,4){\line(0,-1){3}}
\put(5,0){\line(0,1){3}} \put(5,0){\line(1,1){1}}
\put(6,4){\line(-1,-1){1}} \put(6,4){\line(0,-1){3}}
\put(9,0){\line(0,1){3}} \put(9,0){\line(1,1){1}}
\put(10,4){\line(-1,-1){1}} \put(10,4){\line(0,-1){3}}
\put(11,0){\line(0,1){3}} \put(11,0){\line(1,1){1}}
\put(12,4){\line(-1,-1){1}} \put(12,4){\line(0,-1){3}}
\put(1,0){\line(1,0){5}}\put(6,0){\line(1,0){0.2}}\put(6.4,0){\line(1,0){0.2}}\put(6.8,0){\line(1,0){0.2}}
\put(7.2,0){\line(1,0){0.2}}\put(7.6,0){\line(1,0){0.2}}\put(8,0){\line(1,0){3}}
\put(2,1){\line(1,0){5}}\put(7,1){\line(1,0){0.2}}\put(7.4,1){\line(1,0){0.2}}\put(7.8,1){\line(1,0){0.2}}
\put(8.2,1){\line(1,0){0.2}}\put(8.6,1){\line(1,0){0.2}}\put(9,1){\line(1,0){3}}
\put(2,4){\line(1,0){5}}\put(7,4){\line(1,0){0.2}}\put(7.4,4){\line(1,0){0.2}}\put(7.8,4){\line(1,0){0.2}}
\put(8.2,4){\line(1,0){0.2}}\put(8.6,4){\line(1,0){0.2}}\put(9,4){\line(1,0){3}}
\put(1,3){\line(1,0){5}}\put(6,3){\line(1,0){0.2}}\put(6.4,3){\line(1,0){0.2}}\put(6.8,3){\line(1,0){0.2}}
\put(7.2,3){\line(1,0){0.2}}\put(7.6,3){\line(1,0){0.2}}\put(8,3){\line(1,0){3}}
\qbezier(1,0)(6,-3)(11,0) \qbezier(1,3)(6,6)(11,3)
\qbezier(2,1)(7,-2)(12,1) \qbezier(2,4)(7,7)(12,4)
\put(1,0){\oval(0.15,0.1)} \put(1,3){\oval(0.15,0.1)}
\put(2,1){\oval(0.15,0.1)} \put(2,4){\oval(0.15,0.1)}
\put(3,0){\oval(0.15,0.1)} \put(3,3){\oval(0.15,0.1)}
\put(4,1){\oval(0.15,0.1)} \put(4,4){\oval(0.15,0.1)}
\put(5,0){\oval(0.15,0.1)} \put(5,3){\oval(0.15,0.1)}
\put(6,1){\oval(0.15,0.1)} \put(6,4){\oval(0.15,0.1)}
\put(9,0){\oval(0.15,0.1)} \put(9,3){\oval(0.15,0.1)}
\put(10,1){\oval(0.15,0.1)} \put(10,4){\oval(0.15,0.1)}
\put(11,0){\oval(0.15,0.1)} \put(11,3){\oval(0.15,0.1)}
\put(12,1){\oval(0.15,0.1)} \put(12,4){\oval(0.15,0.1)}
\put(1,-0.3){$x_0^0$} \put(2,0.7){$x_1^0$} \put(2,4.3){$x_2^0$}
\put(1,3.3){$x_3^0$} \put(3,-0.3){$x_0^1$} \put(4,0.7){$x_1^1$}
\put(4,4.3){$x_2^1$} \put(3,3.3){$x_3^1$} \put(5,-0.3){$x_0^2$}
\put(6,0.7){$x_1^2$} \put(6,4.3){$x_2^2$} \put(5,3.3){$x_3^2$}
\put(9,-0.3){$x_0^{n-2}$}
\put(10,4.3){$x_2^{n-2}$}\put(10,0.7){$x_1^{n-2}$}
\put(9,3.3){$x_3^{n-2}$} \put(11,-0.3){$x_0^{n-1}$}
\put(12,4.3){$x_2^{n-1}$} \put(12,0.7){$x_1^{n-1}$}
\put(11,3.3){$x_3^{n-1}$} \put(3.8,-2.5){Fig. 1. Graph $C_4\times
C_n$.}
\end{picture}

\vspace{2cm}
\section{Preliminaries}

Let $m$ be a positive integer. Denote
$\alpha(m)=\frac{m+2+\sqrt{m^2+4m}}{2}$,\,
$\beta(m)=\frac{m+2-\sqrt{m^2+4m}}{2}$,
$u_p(m)=\frac{1}{\alpha(m)-\beta(m)}\left(\alpha^p(m)-\beta^p(m)\right)$,
$v_p(m)=\alpha^p(m)+\beta^p(m)$, for $p\in R$.

By the following proposition 2.1, it is easy to see that for every
integer $p\ge 0$, $u_p(m)$ and $v_p(m)$ are integral. The
propositions 2.1 and 2.2 can be easily proved by induction.

\begin{prop}
 If $p$ is integral, then
$$\left\{\begin{array}{ll}
u_p(m)=(m+2)u_{p-1}(m)-u_{p-2}(m),\\
v_p(m)=(m+2)v_{p-1}(m)-v_{p-2}(m),\\
\end{array}\right.\eqno(2.1)$$
with initial values
$$\left\{\begin{array}{cc}
u_0(m)=0,& u_1(m)=1,\\
v_0(m)=2,& \ \ v_1(m)=m+2.\\
\end{array}\right.\eqno(2.2)$$
And if $q\geq 0$ is another integer, then
$u_{pq}(m)=v_{p(q-1)}(m)u_p(m)+u_{p(q-2)}(m)$.
\end{prop}

\begin{prop}
If $p$ is a nonnegative integer, then\\
$\bullet\quad u_p(m)\equiv p\ (\mod\ m),\  \ v_p(m)\equiv 2\ (\mod\ m);\hfill (2.3)$\\
$\bullet\quad v_{2p}(m)=m(m+4)u^2_{p}(m)+2;\hfill (2.4)$\\
$\bullet\quad u_{pq}(m)=\left\{\begin{array}{ll}
V_q(m)u_p(m), &\mbox{if}\quad q \mbox{ is even,}\\
V'_q(m)u_p(m),& \mbox{if}\quad q \mbox{ is odd,}
\end{array}
\right.\hfill (2.5)$\\
where $$V_q(m)=\sum\limits_{0<2i\leq q}v_{p(q+1-2i)}(m),\quad
V'_q(m)=\left(\sum\limits_{0<2i\leq q+1}v_{p(q+1-2i)}(m)\right)-1.\eqno(2.6)$$
\end{prop}

If $n$ is a positive integer of the form $p_1^{t_1}\cdots p_k^{t_k}$
where the $p_i'$s are distinct primes, then let $T_{p_i}(n)$  denote
$t_i$. Let $e_n=u_n(2)$,\, $f_n=u_n(4)$.

\begin{prop}
Let $T_2(n)=t_2,\, T_3(n)=t_3$, for
$n\geq 2$. Then we have  $T_2(e_n)=\left\{\begin{array}{cl}
  0,     & \mbox{if}\quad t_2=0,\\
  t_2+1, & \mbox{if}\quad t_2>0;\end{array}\right.$
$T_2(f_n)=t_2;\,  T_3(e_n)=t_3$; and
$T_3(f_n)=\left\{\begin{array}{cl}
  0,   & \mbox{if}\quad t_2=0, \\
  t_3+1, & \mbox{if}\quad t_2>0.
\end{array}
\right.$
\end{prop}

\begin{proof}
Let $n=2^{t_2}q$, where $q$ is odd.\\
\indent By (2.5), $e_n=V'_q(2)e_{2^{t_2}}$ and  $f_n=V'_q(4)f_{2^{t_2}}$.
By (2.3),  $v_p(2)$ and $v_p(4)$ are even for every $p$ and then from (2.6) we have that $V'_q(2)$ and $V'_q(4)$ are odd.
Thus $T_2(e_n)=T_2(e_{2^{t_2}})$ and $T_2(f_n)=T_2(f_{2^{t_2}})$. If $t_2=0$, then $T_2(e_{2^{t_2}})=T_2(e_1)=0$ and
$T_2(f_{2^{t_2}})=T_2(f_1)=0$. Now we prove by induction on $t_2>0$ that
$T_2(e_{2^{t_2}})=t_2+1$  and $T_2(f_{2^{t_2}})=t_2$.
This is valid if $t_2=1$.  Since from (2.4), (2.5) and (2.6) it follows that
$e_{2^{t_2}}=v_{2^{t_2-1}}(2)e_{2^{t_2-1}}=(12e_{2^{t_2-2}}^2+2)e_{2^{t_2-1}}$ and
$f_{2^{t_2}}=v_{2^{t_2-1}}(4)f_{2^{t_2-1}}=(32f_{2^{t_2-2}}^2+2)f_{2^{t_2-1}}$, then
by the induction hypothesis we have that
$T_2(e_{2^{t_2}})=T_2(12e_{2^{t_2-2}}^2+2)+T_2(e_{2^{t_2-1}})=1+t_2$ and
$T_2(f_{2^{t_2}})=T_2(32f_{2^{t_2-2}}^2+2)+T_2(f_{2^{t_2-1}})=1+t_2-1=t_2$. Thus
$T_2(e_n)=t_2+1$ and $T_2(f_n)=t_2$.\\
\indent Let $n=3^{t_3}\gamma$, where $3\nmid \gamma$.\\
\indent By (2.5), $e_n=\left\{
\begin{array}{ll}
V'_{\gamma}(2)e_{{3^{t_3}}}, & \mbox{if}\; 2\nmid\gamma,\\
V_{\gamma}(2)e_{{3^{t_3}}},  & \mbox{if}\; 2\mid\gamma.
\end{array}\right.$ Note that
$v_{n}(2)=4v_{n-1}(2)-v_{n-2}(2)\equiv v_{n-1}(2)-v_{n-2}(2)\equiv -v_{n-3}\, (\mod 3)$,\,
$v_0(2)=2$,\, $v_1(2)=4\equiv1 (\mod 3)$, $v_2(2)=14\equiv2(\mod 3)$,\, $v_3(2)=52\equiv1(\mod3)$,\,
$v_4(2)=194\equiv2(\mod3)$,\, $v_5(2)=724\equiv1$\, $(\mod3)$.
Then it is not difficult to see that if $2\nmid \gamma$ then
$v_{3^{t_3}(\gamma+1-2i)}(2)\equiv 2\ (\mod 3)$; if $2\mid \gamma$ then $v_{3^{t_3}(\gamma+1-2i)}(2)\equiv 1\ (\mod 3)$.
Hence, if $2\nmid\gamma$, then $V'_{\gamma}(2)\equiv 2\times \frac{\gamma+1}{2}-1=\gamma\ (\mod 3)$;
if $2\mid\gamma$, then $V_{\gamma}(2)\equiv \frac{\gamma}{2} (\mod 3)$. It follows that neither
$V_{\gamma}(2)$ ($\gamma$ is even) nor $V_{\gamma}'(2)$ ($\gamma$ is odd) contains the divisor 3, and hence $T_3(e_n)=T_3(e_{3^{t_3}})$.
Now we prove by induction on $t_3$ that $T_3(e_{3^{t_3}})=t_3$.
It is valid if $t_3=0$, or 1. Since from (2.4) and (2.5) we have that
$e_{3^{t_3}}=(v_{3^{t_3-1}\cdot 2}(2)+v_0(2)-1)e_{3^{t_3-1}}=(12e^2_{3^{t_3-1}}+3)e_{3^{t_3-1}}$. So
by the induction hypothesis we have
$T_3(e_{3^{t_3}})=T_3(12e^2_{3^{t_3-1}}+3)+T_3(e_{3^{t_3-1}})=1+t_3-1=t_3$. Thus
$T_3(e_n)=t_3$.\\
\indent If $t_2=0$, namely $n$ is odd, then we have
$f_n=6f_{n-1}-f_{n-2}\equiv -f_{n-2}\equiv\cdots \equiv (-1)^{\frac{n-1}{2}}f_1\,(\mod 3)$.
Note that $f_1=1$, so $T_3(f_n)=0$.\\
\indent If $t_2>0$, namely $n$ is even, then we can write $n=3^{t_3}\cdot2\epsilon$, where $3\nmid \epsilon$.
By (2.5), $f_n=\left\{
\begin{array}{ll}
V'_{\epsilon}(4)f_{{2\cdot3^{t_3}}}, & \mbox{if}\, 2\nmid\epsilon,\\
V_{\epsilon}(4)f_{{2\cdot3^{t_3}}},  & \mbox{if}\, 2\mid\epsilon.
\end{array}\right.$
By (2.1), $v_{n}(4)=6v_{n-1}(4)-v_{n-2}(4)\equiv -v_{n-2}(4)\equiv\cdots\equiv(-1)^{\frac{n}{2}}v_0(4)=(-1)^\frac{n}{2}2(\mod 3)$.
Then from (2.6) we have that if
$2\nmid\epsilon$,$V'_{\epsilon}(4)\equiv2\times \frac{\epsilon+1}{2}-1=\epsilon\, (\mod 3)$;
if $2\mid\epsilon$, then $V_{\epsilon}(4)\equiv(-2)\times \frac{\epsilon}{2}=-\epsilon\, (\mod 3)$.
Thus neither $V'_{\epsilon}(4)$ ($\epsilon$ is odd) nor $V_{\epsilon}(4)$ ($\epsilon$ is even) is divisible by 3.
So $T_3(f_n)=T_3(f_{2\cdot3^{t_3}})$.
Now we prove by induction on $t_3$ that $T_3(f_{2\cdot3^{t_3}})=t_3+1$.
If  $t_3=0$ or 1, we have $f_2=6$ and $f_6=6930$ respectively, so it is valid.
Since from (2.4), (2.5) and (2.6) it follows that
$f_{2\cdot3^{t_3}}=(v_{2\cdot3^{t_3-1}\cdot 2}(4)+v_0(4)-1)f_{2\cdot3^{t_3-1}}=
(32f^2_{2\cdot3^{t_3-1}}+3)f_{2\cdot3^{t_3-1}}$,
then by the induction hypothesis we have
$T_3(f_{2\cdot3^{t_3}})=T_3(32f^2_{2\cdot3^{t_3-1}}+3)+T_3(f_{2\cdot3^{t_3-1}})=1+t_3$.
\end{proof}

\section{System of relations for the cokernel of the Laplacian on $C_4\times C_n$}

Now we work on the system of relations of the cokernel of
the Laplacian of $C_4\times C_n$. Let
$x_j^i=(0,\cdots,0,1,0,\cdots,0)\in Z^{4n}$, whose unique nonzero 1
is in the position corresponding to vertex $v_j^i$. It follows from
the relations of coker$L(C_4\times C_n)$ that we can get the system
of equations:
$$
4x_j^i-(x_{j+1}^i+x_{j-1}^i)-x_j^{i+1}-x_j^{i-1}=0,\ \ j\in Z_4,\
i\in Z_n.\eqno(3.1)
$$

\begin{lemma}
There are three sequences of integral numbers $(a_i)_{i\ge 0}$,
$(b_i)_{i\ge 0}$, $(c_i)_{i\ge 0}$ such that
$$x_j^i=a_ix_j^1+b_i(x_{j+1}^1+x_{j-1}^1)+c_ix_{j+2}^1-a_{i-1}x_j^0-b_{i-1}(x_{j+1}^0+x_{j-1}^0)-c_{i-1}x_{j+2}^0,\eqno(3.2)$$
where $j\in Z_4,\ 1\leq i\leq n$. Moreover, the numbers in the above
sequences have recurrence relations and initial conditions as
follows
$$\left\{
\begin{array}{ll}
a_i=\frac{1}{4}(i+u_i(4)+2u_i(2)),\quad\quad\quad\quad\quad\quad\;(i\ge 0),\\
b_i=\frac{1}{4}(i-u_i(4)),\quad\quad\quad\quad\quad\quad\quad\quad\quad\quad\;(i\ge 0),\\
c_i=\frac{1}{4}(i+u_i(4)-2u_i(2)),\;\quad\quad\quad\quad\quad\quad\;(i\ge 0).\\
\end{array}\right.\eqno(3.3)
$$
\end{lemma}

\begin{proof}
From (3.1), it follows that
$$x_j^{i+1}=4x_j^i-(x_{j+1}^i+x_{j-1}^i)-x_j^{i-1},\; \mbox{for}\, j\in Z_4,\, 2\leq i\leq n-1.   \eqno(3.4) $$
This lemma is valid for the cases of $i=1,\,2$. Suppose that
$x_j^l=a_lx_j^1+b_l(x_{j+1}^1+x_{j-1}^1)+c_lx_{j+2}^1-a_{l-1}x_j^0-b_{l-1}(x_{j+1}^0+x_{j-1}^0)-c_{l-1}x_{j+2}^0$,
for $l\leq h$. Then from the induction assumption and  equation
(3.4), it follows that
$x_j^{h+1}=4x_j^h-(x_{j+1}^h+x_{j-1}^h)-x_j^{h-1}$
$=4\Big{(}a_hx_j^1+b_h(x_{j+1}^1+x_{j-1}^1)+c_hx_{j+2}^1
-a_{h-1}x_j^0-b_{h-1}(x_{j+1}^0+x_{j-1}^0)-c_{h-1}x_{j+2}^0\Big{)}
-\Big{(}a_hx_{j+1}^1+b_h(x_{j+2}^1+x_{j}^1)
+c_hx_{j-1}^1-a_{h-1}x_{j+1}^0-b_{h-1}(x_{j+2}^0+x_{j}^0)-c_{h-1}x_{j-1}^0\Big{)}
-\Big{(}a_hx_{j-1}^1+b_h(x_{j}^1+x_{j-2}^1)+c_hx_{j+1}^1-a_{h-1}x_{j-1}^0-b_{h-1}(x_{j}^0+x_{j-2}^0)-
          c_{h-1}x_{j+1}^0\Big{)}-\Big{(}a_{h-1}x_j^1+b_{h-1}(x_{j+1}^1+x_{j-1}^1)
          +c_{h-1}x_{j+2}^1-a_{h-2}x_j^0-b_{h-2}(x_{j+1}^0+x_{j-1}^0)-c_{h-2}x_{j+2}^0\Big{)}$
$=\Big{(}4a_h-2b_h-a_{h-1}\Big{)}x_j^1+\Big{(}4b_h-a_h-c_h-b_{h-1}\Big{)}\Big{(}x_{j+1}^1+x_{j-1}^1\Big{)}
+\Big{(}4c_h-2b_h-c_{h-1}\Big{)}x_{j+2}^1
          -\Big{(}4a_{h-1}-2b_{h-1}-a_{h-2}\Big{)}x_j^0-\Big{(}4b_{h-1}-a_{h-1}-c_{h-1}-b_{h-2}\Big{)}
          \Big{(}x_{j+1}^0+x_{j-1}^0\Big{)}
          -\Big{(}4c_{h-1}-2b_{h-1}-c_{h-2}\Big{)}x_{j+2}^0$
$=a_{h+1}x_j^1+b_{h+1}(x_{j+1}^1+x_{j-1}^1)+c_{h+1}x_{j+2}^1-a_hx_j^0-b_h(x_{j+1}^0+x_{j-1}^0)-c_hx_{j+2}^0.$

Thus (3.2) holds by induction.\\
\indent From the process of induction just now, it is easy to see that
$$\left\{\begin{array}{ll}
a_{i+1}=4a_i-2b_i-a_{i-1},\\
b_{i+1}=4b_i-(a_i+c_i)-b_{i-1},\\
c_{i+1}=4c_i-2b_i-c_{i-1},
\end{array}\right.\eqno(3.5)$$
for $ i\geq 1$. Let $\tau_i=a_i+c_i$ and $\eta_i=a_i-c_i$. After a short
calculation, we can get
$$\left\{\begin{array}{ll}
\eta_{i+1}=4\eta_i-\eta_{i-1},\\
\eta_0=0,\quad \eta_1=1;\\
\tau_{i+2}-8\tau_{i+1}+14\tau_i-8\tau_{i-1}+\tau_{i-2}=0,\\
\tau_0=0, \quad \tau_1=1.
\end{array}\right.$$
By proposition 2.1, we have $\eta_i=u_i(2)=e_i$. Let $\phi_i=2\tau_i-i$, then one
can verify that $\phi_{i+2}=6\phi_{i+1}-\phi_i$, with $\phi_0=0$ and $\phi_1=1$.
Immediately, $\phi_i=u_i(4)=f_i$, and then
$\tau_i=\frac{1}{2}(i+u_i(4))$. Now  the equalities in (3.3) can been verified directly.
\end{proof}

We know from lemma 3.1 that the systerm of equation (3.2) has at
most 8 generators, i.e., each $x_j^i$ can be expressed in terms of
$x_0^0,x_1^0,x_2^0,x_3^0,x_0^1,x_1^1,x_2^1,\\x_3^1.$ So there are at
least $4n-8$ diagonal entries of the Smith normal form of $L(G)$ are
equal to 1, however the remaining invariant factors of
coker($C_4\times C_n$) hide inside the relations matrix induced by
$x_0^0,x_1^0,x_2^0,x_3^0,x_0^1,x_1^1,x_2^1,x_3^1$.

Let $Y=(x_0^1,\ x_1^1,\ x_2^1,\ x_3^1,\ x_0^0,\ x_1^0,\ x_2^0,\
x_3^0)^t$, $A_n=\left(\begin{array}{cccc}
a_n & b_n & c_n & b_n\\
b_n & a_n & b_n & c_n\\
c_n & b_n & a_n & b_n\\
b_n & c_n & b_n & a_n
\end{array} \right)$ and $M=\left(\begin{array}{cc}
A_{n+1} & -A_n\\
A_n & -A_{n-1}
\end{array} \right).$  From (3.2) and the cyclic structure of $C_4\times C_n$, we
have
$$\left\{\begin{array}{ll}
x_j^0=x_j^n=a_{n}x_j^1+b_{n}(x_{j+1}^1+x_{j-1}^1)+c_{n}x_{j+2}^1-a_{n-1}x_j^0\\
\hspace{1.7cm}-b_{n-1}(x_{j+1}^0+x_{j-1}^0)-c_{n-1}x_{j+2}^0,\\
x_j^1=x_j^{n+1}=a_{n+1}x_j^1+b_{n+1}(x_{j+1}^1+x_{j-1}^1)+c_{n+1}x_{j+2}^1-a_nx_j^0\\
\hspace{1.7cm}-b_n(x_{j+1}^0+x_{j-1}^0)-c_nx_{j+2}^0,
\end{array}\right.$$
where  $0\leq j\le 3$. Therefore
$$(M-I)Y=0.\eqno(3.6)$$

From the argument above, we know that one can reduce $L(G)$
to  $I_{4n-8}\oplus (M-I)$ by performing some row and column
operations up to equivalence.  Now
we only need to evaluate the SNF of $M-I$.

\section{Analysis of the coefficients of the Smith normal form of $M-I$}

If we multiply  the last 4 rows of $M-I$ by $-1$, then we
have that
$$\left(\begin{array}{cc}
A_{n+1}-I_4& -A_n\\
A_n        &-A_{n-1}-I_4
\end{array}\right)\sim\left(\begin{array}{cc}
A_{n+1}-I_4 & -A_n\\
-A_n        & A_{n-1}+I_4
\end{array}\right).\eqno (4.1)$$
\indent From lemma 3.1, one can verify that
$a_{i+1}+c_{i+1}+2b_{i+1}=a_{i}+c_{i}+2b_{i}+1$, for each $i\in N$,
 and it results that  each line sum of the  right matrix  of (4.1)
is equal to 0. Immediately, we have the following lemma.

\begin{lemma}
$M-I\sim (0)\oplus M_1$, where $M_1$ is the
submatrix of $M-I$ resulting from the deletion of the first
row and column.
\end{lemma}

Let $h_n=e_n+e_{n+1}$, $g_n=f_n+f_{n+1}$, $p_i=e_{i}+e_{n-i}$, $q_i=f_{i}+f_{n-i}$, and let
$$L_1=\small{\begin{pmatrix}
    0 & 0 & 0 & 1 & 1 & 1 & 1 \\
    1 & 2 & 1 & -1 & -1 & -1 & -1 \\
    0 & 0 & 0 & -1 & 0 & 1 & 0 \\
    0 & 1 & 0 & 0 & 0 & 0 & 0 \\
    0 & 0 & 0 & 0 & 0 & 1 & 0 \\
    1 & 1 & 0 & 0 & 0 & 0 & 0 \\
    0 & 0 & 0 & 0 & 1 & 1 & 0 \\
  \end{pmatrix}},\; R_1=
  \small{\begin{pmatrix}
    -1 & -1 & 0 & 1 & 0 & 1 & 0 \\
    0 & 1 & 0 & 0 & 0 & 0 & 0 \\
    -1 & 0 & 0 & 0 & 0 & -1 & 0 \\
    -1 & 0 & 0 & 0 & 0 & 0 & 0 \\
    0 & 0 & 1 & 0 & -1 & 0 & -1 \\
    -1 & 0 & -1 & 0 & 0 & 0 & 0 \\
    0 & 0 & 0 & 0 & 0 & 0 & 1 \\
  \end{pmatrix}}.$$
\indent Then one can check  that $L_1$ and $R_1$ are unimodular
matrices and
$$ L_1M_1R_1=
\begin{pmatrix}
  0                      &    0                     & 0                        &   n                  &   n                  & 0 & 0\\
  0                      &   p_{-1}                    & p_0                      &   0                  &   0                  & 0 & 0\\
  0                      &   p_0                    & p_{1}                   &   0                  &   0                  & 0 & 0\\
\frac{q_{-1}+q_0}{2}  &\frac{p_{-1}+q_{-1}}{2}& \frac{p_{0}+q_{0}}{2} &\frac{n-q_{-1}}{4} & \frac{n-q_{0}}{4} &0  &0\\
\frac{q_{0}+q_1}{2}   &\frac{p_0+q_0}{2}      &\frac{p_{1}+q_{1}}{2}&\frac{n-q_{0}}{4}  & \frac{n-q_{1}}{4}&0  &0\\
0           &0    &0     &\frac{n+p_{-1}}{2}     &\frac{n+p_{0}}{2}       &p_{-1}  &p_0\\
0           &0    &0     &\frac{n+p_{0}}{2}
&\frac{n+p_{1}}{2}&p_0   &p_{1}
\end{pmatrix}.$$
Putting $m=2$ and $4$, then it follows from  proposition 2.1 that
$$\left\{\begin{array}{ll}
p_{i+1}=4p_i-p_{i-1},\\
q_{i+1}=6q_i-q_{i-1}.
\end{array}\right.\eqno(4.2)$$
\indent Let $M_2=L_1M_1R_1$ and $ U=
  \begin{pmatrix}
    1 & 0 & 0 & 0 & 0 & 0 & 0 \\
    0 & 0 & 1 & 0 & 0 & 0 & 0 \\
    0 & -1 & 4 & 0 & 0 & 0 & 0 \\
    0 & 0 & 0 & 0 & 1 & 0 & 0 \\
    -1 & 0 & -1 & -1 & 6 & 0 & 0 \\
    0 & 0 & 0 & 0 & 0 & 0 & 1 \\
    -1 & 0 & 0 & 0 & 0 & -1 & 4 \\
  \end{pmatrix}$.\\
Then by (4.2) we have
$$U^iM_2=
\small{\begin{pmatrix}
  0                      &    0                     & 0                        &   n                  &   n                  & 0 & 0\\
  0                      &   p_{i-1}                    & p_i                      &   0                  &   0                  & 0 & 0\\
  0                      &   p_i                    & p_{i-1}                   &   0                  &   0                  & 0 & 0\\
\frac{q_{i-1}+q_i}{2}  &\frac{p_{i-1}+q_{i-1}}{2}& \frac{p_{i}+q_{i}}{2} &\frac{n-q_{i-1}}{4} & \frac{n-q_{i}}{4} &0  &0\\
\frac{q_{i}+q_{i+1}}{2}   &\frac{p_i+q_i}{2}      &\frac{p_{i+1}+q_{i+1}}{2}&\frac{n-q_{i}}{4}  & \frac{n-q_{i+1}}{4}&0  &0\\
0           &0    &0     &\frac{n+p_{i-1}}{2}     &\frac{n+p_{i}}{2}       &p_{i-1}  &p_i\\
0           &0    &0     &\frac{n+p_{0}}{2}
&\frac{n+p_{i+1}}{2}  &p_i     &p_{i+1}
\end{pmatrix}}.\eqno(4.3)
$$

Now we distinguish two cases.

\noindent{\bf Case 1 $n=2s+1$  odd.}\\
\indent In this case, by (4.2) one can verify that $$\left(\begin{array}{cc}
p_s & p_{s+1}\\
p_{s+1} & p_{s+2}
\end{array}\right)=\left(\begin{array}{cc}
h_s & h_s\\
h_s & 3h_s
\end{array}\right),\quad \left(\begin{array}{cc}
q_s & q_{s+1}\\
q_{s+1} & q_{s+2}
\end{array}\right)\left(\begin{array}{cc}
g_s & g_s\\
g_s & 5g_s
\end{array}\right).\eqno(4.4) $$
Let $i=s+1$ in (4.3), then by (4.4) we have
$$U^{s+1}M_2=
\left(
  \begin{array}{ccccccc}
    0 & 0 & 0 & n & n & 0 & 0 \\
    0 & h_s & h_s & 0 & 0 & 0 & 0 \\
    0 & h_s & 3h_s & 0 & 0 & 0 & 0 \\
    g_s & \frac{g_s+h_s}{2} & \frac{g_s+h_s}{2} & \frac{n-g_s}{4} & \frac{n-g_s}{4} & 0 & 0 \\
    3g_s & \frac{g_s+h_s}{2} & \frac{5g_s+3h_s}{2} & \frac{n-g_s}{4} & \frac{n-g_s}{4} & 0 & 0 \\
    0 & 0 & 0 & \frac{n+h_s}{2} & \frac{n+h_s}{2} & h_s & h_s \\
    0 & 0 & 0 & \frac{n+h_s}{2} & \frac{n+3h_s}{2} & h_s & 3h_s \\
  \end{array}
\right).
$$
Let $$ L_2= \left(
  \begin{array}{ccccccc}
    0 & -1 & 1 &   0    & 0  & 0  & 0 \\
    0 & 0  & 0 &   0    & 0  & 1  & -1 \\
    0 & 0  & 0 &  -1    & 1  & 0  & 0 \\
    1 & 0  & 0 &   0    & 0  & 0  & 0 \\
    0 & 1  & 0 &   0    & 0  & 0  & 0 \\
    0 & 0  & 0 &   0    & 0  & 1  & 0 \\
    0 & 0  & 0 &   1    & 0  & 0  & 0 \\
  \end{array}
\right),\quad R_2= \left(
  \begin{array}{ccccccc}
    0  & 0  & 0  & 0 & 0 & 0 & 1 \\
    0  & -1 & 0  & 0 & 1 & 0 & 0 \\
    0  & 1  & 0  & 0 & 0 & 0 & 0 \\
    1  & -2 & 0  & 1 & 0 & 0 & -2 \\
    -1 & 2  & 0  & 0 & 0 & 0 & 2 \\
    0  & 1  & 1  & 0 & 0 & 1 & 1 \\
    0  & -1 & -1 & 0 & 0 & 0 & -1 \\
  \end{array}
\right).$$ It is clear that $L_2$ and $R_2$ are unimodular matrices.  By a direct calculation, we get
$$ L_2U^{s+1}M_2R_2= X\oplus Y,\eqno(4.5)$$
 where
$X=\left(
  \begin{array}{ccc}
    0 & 2h_s & 0 \\
    h_s & 0 & 2h_s \\
    g_s & h_s & 0 \\
  \end{array}
\right)$ and $Y=\left(
  \begin{array}{cccc}
    n & 0 & 0 & 0 \\
    0 & h_s & 0 & 0 \\
    \frac{n+h_s}{2} & 0 & h_s & 0 \\
    \frac{n-g_s}{4} & \frac{h_s+g_s}{2} & 0 & g_s \\
  \end{array}\right)$.

Using the standard method for calculating the determinant factors we have that
$$\mbox{SNF}(X)=\mbox{diag}\left((h_s,g_s),h_s, \frac{4h_sg_s}{(h_s,\ g_s)}\right)$$
and
$$\mbox{SNF}(Y)=\mbox{diag}\left((n,\ h_s,\ g_s),\frac{(n,\ h_s)(h_s,\ g_s)}{(n,\
h_s,\ g_s)}, \frac{h_s(nh_s,
ng_s,h_sg_s)}{(n,h_s)(h_s,g_s)},\frac{nh_sg_s}{(nh_s,\ ng_s,\
h_sg_s)}\right).$$ \indent From above, now it is easy to see that in
this case $\mbox{SNF}(M_1)=\mbox{SNF}(M_2)=$\\
$$\mbox{diag}\left((n,\, h_s,\, g_s),\, (h_s,\, g_s),\, \frac{(n,\,
h_s)(h_s,\, g_s)}{(n,\, h_s,\, g_s)},\, h_s, \frac{h_s(nh_s,\
ng_s,\, h_sg_s)}{(n,\ h_s)(h_s,\ g_s)},\ \frac{h_sg_s}{(h_s,\
g_s)},\,
 \frac{4nh_sg_s}{(nh_s,\ ng_s,\ h_sg_s)}\right).$$

\noindent{\bf Case 2 $n=2s$  even. }\\
\indent In this case, by (4.2) one can verify that
$$\left(\begin{array}{cc}
p_s & p_{s+1}\\
p_{s+1} & p_{s+2}
\end{array}\right)=\left(\begin{array}{cc}
2e_s & 4e_s\\
4e_s & 14e_s
\end{array}\right),\quad \left(\begin{array}{ll}
q_s & q_{s+1}\\
q_{s+1} & q_{s+2}
\end{array}\right)=
\left(\begin{array}{ll}
2f_s & 6f_s\\
6f_s & 34f_s
\end{array}\right).$$

Apply (4.3), we have
$$U^{s+1}M_2=\left(
  \begin{array}{ccccccc}
    0 & 0 & 0 & 2s & 2s & 0 & 0 \\
    0 & 2e_s & 4e_s & 0 & 0 & 0 & 0 \\
    0 & 4e_s & 14e_s & 0 & 0 & 0 & 0 \\
    4f_s & f_s+e_s & 3f_s+2e_s & \frac{s-f_s}{2} & \frac{s-3f_s}{2} & 0 & 0 \\
    20f_s & 3f_s+2e_s & 17f_s+7e_s & \frac{s-3f_s}{2} & \frac{s-17f_s}{2} & 0 & 0 \\
    0 & 0 & 0 & s+e_s & s+2e_s & 2e_s & 4e_s \\
    0 & 0 & 0 & s+2e_s & s+7e_s & 4e_S & 14e_s \\
  \end{array}\right).$$
Let
$$L_3=\left(
  \begin{array}{ccccccc}
    1  & 0  & 0  & 0 & 0  & 0  & 0 \\
    0  & -1 & 1  & 0 & 0  & 0  & 0 \\
    1  & 0  & 0  & 0 & 0  & -1 & -1 \\
    -1 & -4 & 1  & 7 & -1 & 0  & 0 \\
    0  & 5  & -4 & 0 & 0  & 0  & 0 \\
    0  & 0  & 0  & 0 & 0  & 1  & 0 \\
    0  & 2  & -2 & 1 & 1  & 0  & 0 \\
  \end{array}\right),\quad
R_3=\left(
  \begin{array}{ccccccc}
    0  & -2 & 0 & 1 &-2 & 0 & 0 \\
    0  & 6  & 0 & 0 & 5 & 0 & 0 \\
    0  & -1 & 0 & 0 &-1 & 0 & 0 \\
    2  & 6  & 0 &-4 & 6 & 1 & 2 \\
    -1 & -6 & 0 & 4 &-1 &-1 &-2 \\
    0  & -3 & 2 & 2 &-3 & 1 &-1 \\
    0  & 3  &-1 &-2 & 3 & 0 & 1 \\
  \end{array}\right).$$
Then we have
$$L_3U^{s+1}M_2R_3=\left(
  \begin{array}{ccccccc}
    2s                 & 0        & 0    & 0    & 0    & 0    & 0 \\
    0                  & 2e_s     & 0    & 0    & 0    & 0    & 0 \\
    3e_s               & 0        & 6e_s & 0    & 0    & 0    & 0 \\
    s-2f_s             & e_s+4f_s & 0    & 8f_s & 0    & 0    & 0 \\
    0                  & 0        & 0    & 0    & 6e_s & 0    & 0 \\
    s                  & 0        & 0    & 0    & 0    & e_s  & 0 \\
    \frac{1}{2}(f_s+s) & f_s      & 0    & 0    & 3e_s & f_s  & 2f_s \\
  \end{array}\right).\eqno (4.6)$$
\indent Let $M_3$ denote the matrix on the right side of (4.6). If we can further reduce $M_3$ to  the direct product of
some small matrices as in the above case of $n$ being odd, then the
calculation will become easier. Unfortunately, we can not achieve
it.

Let
$$
M_3'=\left(
  \begin{array}{ccccccc}
    s                 & 0        & 0    & 0    & 0    & 0    & 0 \\
    0                  & e_s     & 0    & 0    & 0    & 0    & 0 \\
    3e_s               & 0        & 3e_s & 0    & 0    & 0    & 0 \\
    s-2f_s             & e_s+4f_s & 0    & f_s & 0    & 0    & 0 \\
    0                  & 0        & 0    & 0    & 3e_s & 0    & 0 \\
    s                  & 0        & 0    & 0    & 0    & e_s  & 0 \\
    \frac{1}{2}(f_s+s) & f_s      & 0    & 0    & 3e_s & f_s  & f_s \\
  \end{array}
\right),$$
$L_4=\left(
      \begin{array}{ccccccc}
        1 & 0 & 0 & 0 & 0 & 0 & 0 \\
        0 & 0 & 0 & 0 & -1 & 0 & 1 \\
        0 & 1 & 0 & 0 & 0 & 0 & 0 \\
        0 & 0 & 1 & 0 & 0 & 0 & 0 \\
        -1 & -1& 0 & 1 & 0 & 0 & 0 \\
        -1 & 0 & 0 & 0 & 0 & 1 & 0 \\
        0 & 0 & 0 & 0 & 1 & 0 & 0 \\
      \end{array}
    \right),\;
R_4=\left(
      \begin{array}{ccccccc}
        1 & 0 & 0 & 0 & 0 & 0 & 0 \\
        0 & 0 & 1 & 0 & 0 & 0 & 0 \\
        -1 & 0 & 0 & 1 & 0 & 0 & 0 \\
        2 & 0 & -4 & 0 & 1 & 0 & 0 \\
        0 & 0 & 0 & 0 & 0 & 0 & 1 \\
        0 & 0 & 0 & 0 & 0 & 1 & 0 \\
        0 & 1 & -1 & 0 & 0 & -1 & 0 \\
      \end{array}
    \right).$\\
It is clear that $L_4$ and $R_4$ are unimodular matrices and $L_4M_3'R_4=E\oplus F$, where
$$E=\left(
    \begin{array}{cccc}
      s & 0 & 0 & 0 \\
      \frac{1}{2}(s+f_s) & f_s & 0 & 0 \\
      0 & 0 & e_s & 0 \\
      0 & 0 & 0 & 3e_s \\
    \end{array}
  \right),\quad F=\left(
    \begin{array}{ccc}
      f_s & 0 & 0 \\
      0 & e_s & 0 \\
      0 & 0 & 3e_s \\
    \end{array}
  \right).\eqno(4.7)$$

Now we can  compute the determinantal divisors of $E$ and $F$ and furthermore
obtain the  SNF of $M_3'$. Here we directly give the result and omit the details of computation.
However we must say that proposition 2.3 plays an important role in this computation.

$$\mbox{SNF}(E)=\left\{\begin{array}{ll}
\mbox{diag}\left((s,e_s,f_s), \frac{(s, e_s)(e_s, f_s)}{(s, e_s,
f_s)},
\frac{e_s(se_s, sf_s, e_sf_s)}{(s, e_s)(e_s, f_s)}, \frac{3se_sf_s}{(se_s, sf_s, e_sf_s)}\right), &\mbox{if}\, 2\nmid s,\\
\mbox{diag}\left((s, e_s, f_s), \frac{(s, e_s)(e_s, f_s)}{(s, e_s,
f_s)}, \frac{3e_s(se_s, sf_s, e_sf_s)}{(s, e_s)(e_s, f_s)},
\frac{se_sf_s}{(se_s, sf_s, e_sf_s)}\right), &\mbox{if}\, 2\mid s;
\end{array}\right.$$
and
$$\mbox{SNF}(F)=\left\{\begin{array}{ll}
\mbox{diag}\left((e_s, f_s), e_s, \frac{3e_sf_s}{(e_s, f_s)}\right), & \mbox{if}\,  2\nmid s,\\
\mbox{diag}\left((e_s, f_s), 3e_s, \frac{e_sf_s}{(e_s, f_s)}\right), & \mbox{if}\, 2\mid s.
\end{array}\right.$$

Then it is not hard to see that $\mbox{SNF}(M_3')=$
$$\left\{\begin{array}{ll}
\mbox{diag}\left((s, e_s, f_s), (e_s, f_s), \frac{(s, e_s)(e_s,
f_s)}{(s, e_s, f_s)}, e_s,
\frac{e_s(se_s, ef_s, e_sf_s)}{(s, e_s)(e_s, f_s)}, \frac{3e_sf_s}{(e_s,\ f_s)}, \frac{3se_sf_s}{(se_s, ef_s, e_sf_s)}\right), & \mbox{if}\, 2\nmid s,\\
\mbox{diag}\left((s,e_s,f_s), (e_s, f_s),  \frac{(s, e_s)(e_s,
f_s)}{(s, e_s, f_s)}, 3e_s, \frac{3e_s(se_s, sf_s, e_sf_s)}{(s,
e_s)(e_s, f_s)}, \frac{e_sf_s}{(e_s,\ f_s)}, \frac{se_sf_s}{(se_s,
ef_s, e_sf_s)}\right), & \mbox{if}\, 2\mid s.
\end{array}\right.$$

Note that $M_3$ is obtained from $M_3'$ by  multiplying its  rows
1, 2 , 5 by 2,  columns 3, 7 by 2,  column 4 by 8. Then
we have that there are  integers $t_i$  such that $S_i(M_3)=2^{t_i}S_i(M_3')$, for $1\leq i\leq 7$.

\indent{$\bullet$ } $n=2s$ with $s$  odd.\\
\indent It follows from proposition 2.3 that  $2\nmid e_s$ and
$2\nmid f_s$. Moreover, $\Delta_i(M_3')$ is odd and hence
$S_i(M_3')$ is odd. Since  $\det
(M_3[3,4,6,7|1,2,5,6])=-9e_s^3(e_s+4f_s)$ is odd, where
$M_3[3,4,6,7|1,2,5,6]$ is the submatrix that lies in the rows 3, 4,
6, 7 and columns 1, 2, 5, 6 of $M_3$. Thus $t_1=t_2=t_3=t_4=0$. Note
that  every  nonzero element in rows 1, 2, 5,  columns 3, 4, 7 of
$M_3$ is  even and on the main diagonal, so every $5\times5$
submatrix of $M_3$ must contain at least one row and at least one
column of them. Thus $2^2\mid \Delta_5(M_3)$. Since
$\det\left(M_3[1, 3, 4, 6, 7 |1, 2, 3, 5,
6]\right)=36se_s^3(e_s+4f_s)$, then $2^3$ is not its divisor. Thus
$t_5=2$. As above, $2^4\mid\Delta_6(M_3)$, but $\det\left(M_3[1, 3,
4, 5, 6, 7| 1, 2, 3, 5, 6, 7]\right)=-144se_s^3f_s(e_s+4f_s)$, which
is not divisible by $2^5$. So $t_6=4-2=2$. Finally, it is easy to
see that $t_7=8-4=4$. Thus the SNF of $M_3$ here is
$\mbox{diag}\left((s,\, e_s,\, f_s), (e_s,\, f_s), \frac{(s,\,
e_s)(e_s,\, f_s)}{(s,\, e_s,\, f_s)},\,  e_s,\, \frac{4e_s(se_s,\,
sf_s,\, e_sf_s)}{(s,\, e_s)(e_s,\, f_s)},
 \frac{12e_sf_s}{(e_s,\, f_s)}, \frac{48se_sf_s}{(se_s,\, sf_s,\,
 e_sf_s)}\right).$

\indent{ $\bullet$ } $n=2s$ with $s$  even.\\
\indent Let $t=T_2(s)$, then from proposition 2.3, it follows that
$T_2(e_s)=t+1$ and $T_2(f_s)=t$. It is clear that
$S_1(M_3)=S_1(M_3')$, so $t_1=0$. Since
$T_2(\det(M_3[6,7|1,2]))=T_2(sf_s)=2t=T_2(\Delta_2(M_3'))=2t$, then
clearly $t_2=0$. It is not hard to see that the maximal power of 2
contained in each of  the $3\times 3$ minor subdeterminants of $M_3$
is at least $3t+2$, and then we can conclude that
$T_2(\Delta_3(M_3))=3t+2$, since
$\det\left(M_3[4,6,7|1,2,7]\right)=-2sf_s(e_s+4f_s)$ is not
divisible by $2^{3t+3}$. Then
$T_2(S_3(M_3))=T_2(\Delta_3(M_3))-T_2(\Delta_2(M_3))=(3t+2)-2t=t+2$.
So $t_3=T_2(S_3(M_3))-T_2(S_3(M_3'))=(t+2)-t=2$. All the $4\times 4$
minor subdeterminants of $M_3$ contain the divisor $2^{4t+4}$, and
then we can say that $T_2(\Delta_4(M_3))=4t+4$, since
$\det\left(M_3[3,4,6,7|1,2,3,7]]\right)=-12se_sf_s(e_s+4f_s)$ is not
divisible by $2^{4t+5}$. Then
$T_2(S_4(M_3))=T_2(\Delta_4(M_3))-T_2(\Delta_3(M_3)=(4t+4)-(3t+2)=t+2$.
So $t_4=T_2(S_4(M_3))-T_2(S_4(M_3'))=(t+2)-(t+1)=1$. Go on in this
way, we obtain that $t_5=t_6=1$ and $t_7=3$. Thus we get that SNF of
$M_3$ here is $\mbox{diag}\left((s,e_s,f_s),\, (e_s,f_s),\,
\frac{4(s,\, e_s)(e_s,\, f_s)}{(s,\, e_s,\, f_s)},\, 6e_s,\,
\frac{6e_s(se_s,\, sf_s,\, e_sf_s)}{(s,\, e_s)(e_s,\, f_s)},\,
\frac{2e_sf_s}{(e_s,\, f_s)},\, \frac{8se_sf_s}{(se_s,\, sf_s,\,
e_sf_s)}\right).$

\section{Conclusion}

Now we can give the main result as follows.

\begin{theorem}
If $n=2s+1$ odd, then the critical group of $C_4\times C_n$ $(n\ge
3)$ is
$$Z_{(n,\, h_s,\, g_s)}\oplus Z_{(h_s,\, g_s)}\oplus Z_{\frac{(n,\, h_s)(h_s,\, g_s)}{(n,\, h_s,\, g_s)}}
\oplus Z_{h_s}\oplus Z_{\frac{h_s(nh_s, ng_s, h_sg_s)}{(n, h_s)(h_s,
g_s)}} \oplus Z_{\frac{h_sg_s}{(h_s, g_s)}}\oplus
Z_{\frac{4nh_sg_s}{(nh_s, ng_s, h_sg_s)}}.$$ If $n=2s$ with $s$
odd, then the critical group of $C_4\times C_n$ $(n\ge 3)$ is
$$Z_{(s, e_s, f_s)}\oplus Z_{(e_s, f_s)}\oplus Z_{\frac{(s, e_s)(e_s, f_s)}{(s, e_s, f_s)}}
\oplus Z_{e_s}\oplus Z_{\frac{4e_s(se_s, sf_s, e_sf_s)}{(s, e_s)(e_s, f_s)}} \oplus Z_{\frac{12e_sf_s}{(e_s, f_s)}}\oplus
Z_{\frac{48se_sf_s}{(se_s,\, sf_s,\, e_sf_s)}}.$$
If $n=2s$ with $s$  even, then the critical group of $C_4\times C_n$ $(n\ge
3)$ is
$$Z_{(s,\, e_s,\, f_s)}\oplus Z_{(e_s,\, f_s)}\oplus Z_{\frac{4(s,\, e_s)(e_s,\, f_s)}{(s,\, e_s,\, f_s)}}
\oplus Z_{6e_s}\oplus Z_{\frac{6e_s(se_s, sf_s, e_sf_s)}{(s,
e_s)(e_s,\, f_s)}} \oplus Z_{\frac{2e_sf_s}{(e_s, f_s)}}\oplus
Z_{\frac{8se_sf_s}{(se_s, sf_s, e_sf_s)}}.$$
\end{theorem}

\begin{exam}
To give an illustration of theorem 5.1,
we consider the three graphs $C_4\times C_4$, $C_4\times C_5$ and $C_4\times C_6$.
Note that $e_0=0,\, e_1=1,\, e_2=4,\, e_3=15,\, e_4=56,\,  e_5=209,\,  e_6=780$,
$f_0=0,\, f_1=1,\, f_2=6,\, f_3=35,\, f_4=204,\, f_5=1189,\, f_6=6930$.
Then by theorem 5.1 we have that $K(C_4\times C_4)=(Z_2)^2\oplus Z_8\oplus (Z_{24})^3\oplus Z_{96}$;
$K(C_4\times C_5)=(Z_{19})^2\oplus Z_{779}\oplus Z_{15580}$ and
$K(C_4\times C_6)=Z_5\oplus (Z_{15})^2\oplus Z_{60}\oplus Z_{1260}\oplus Z_{5040}$. Maple gives
the identical result.
\end{exam}

Let $H_n(m)=u_n(m)+u_{n+1}(m)$. Clearly, $H_n(2)=h_n$ and
$H_n(4)=g_n$.\\

\begin{theorem}
 If $n_1\mid n_2$, then $K(C_4\times C_{n_1})$ is a subgroup of $K(C_4\times C_{n_2})$.
\end{theorem}

\begin{proof}
We only need to prove that every invariant factor of $K(C_4\times
C_{n_1})$ is
a divisor of the corresponding one of $K(C_4\times C_{n_2})$. We distinguish three cases.\\
\indent $Case$ 1.  $n_1=2s+1$  and $n_2=(2k+1)(2s+1)$.\\
\indent Let $p=2s+1$, $q=2k+1$, then $H_{\frac{\lfloor n_2 \rfloor}{2}}(m)=H_{pk+s}(m)$.
Since $\alpha\beta=1$, then from the definition we can directly verify that $u_{pk+s}(m)=v_{pk}u_s(m)+u_{pk-s}(m)$,
$u_{pk+s+1}(m)=v_{pk}(m)u_{s+1}(m)+u_{pk-s-1}(m)$.
Thus $H_{pk+s}(m)=v_{pk}(m)H_s(m)+H_{pk-s-1}(m)=v_{pk}(m)H_s(m)+H_{p(k-1)+s}=
\cdots=\left(\sum\limits_{i=1}^kv_{ip}(m)+1\right)H_s(m)$.
It means that  $H_s(m)\mid H_{pk+s}(m)$ and hence $h_s\mid h_{pk+s}$,\, $g_s\mid g_{pk+s}$. So every invariant factor
of $K(C_4\times C_{2s+1})$ is a
divisor of the corresponding one of  $K(C_4\times C_{(2k+1)(2s+1)})$.\\
\indent $Case$ 2.\, $n_1=2s+1$ and $n_2=2k(2s+1)$.\\
\indent Since one can verify that
$(u_n(m)+u_{n+1}(m))(u_n(m)-u_{n+1}(m))=-u_{2n+1}(m)$ and $u_{n}(m)=v_p(m)u_{n-p}(m)-u_{n-2p}(m)$,
we have that $H_n(m)\mid u_{2n+1}(m)$ and
if $p\mid n$, then $u_p(m)\mid u_n(m)$.  Thus $H_s(m)\mid u_{n_1}(m)$, and $u_{n_1}(m)\mid u_{kn_1}(m)$.
Then $H_s(m)\mid u_{kn_1}(m)$. It
means that $h_s\mid e_{kn_1}$ and $g_s\mid f_{kn_1}$. So every invariant factor
of $K(C_4\times C_{2s+1})$ is a
divisor of the corresponding one of  $K(C_4\times C_{2k(2s+1)})$.\\
\indent $Case$ 3.\, $n_1=2s$  and $n_2=2ks$.\\
\indent Since $u_s(m)\mid u_{ks}(m)$, then
 $e_s\mid e_{ks}$ and $f_s\mid f_{ks}$. So every invariant factor
of $K(C_4\times C_{2s})$ is a
divisor of the corresponding one of  $K(C_4\times C_{2ks})$.
\end{proof}

\begin{theorem}
The  spanning tree number  of $C_4\times C_n$ ($n\geq 3$)  is
$2^73^2ne_{\frac{n}{2}}^4f_{\frac{n}{2}}^2$, i.e.,
$$\frac{n}{4^{n+1}}\left((\sqrt{3}+1)^{n}-(\sqrt{3}-1)^{n}\right)^4\cdot
\left((\sqrt{2}+1)^{n}-(\sqrt{2}-1)^{n}\right)^2.$$
\end{theorem}

\begin{proof}
We prove this theorem by distinguishing two cases.\\
\indent $Case$\, 1:  $n=2s+1$.\\
\indent A direct calculation shows that
$h_s^4=(e_s+e_{s+1})^4=\frac{1}{4^{n+1}}\left((\sqrt{3}+1)^{n}-(\sqrt{3}-1)^{n}\right)^4$
and
$g_s^2=(f_s+f_{s+1})^2=\frac{1}{4}\left((\sqrt{2}+1)^{n}-(\sqrt{2}-1)^{n}\right)^2.$\\
\indent From (4.3), we know that the spanning tree number of
$C_4\times C_n$ of this case  is $(\det X)\cdot(\det Y)=4nh_s^4g_s^2
=\frac{n}{4^{n+1}}\left((\sqrt{3}+1)^{n}-(\sqrt{3}-1)^{n}\right)^4\cdot
\left((\sqrt{2}+1)^{n}-(\sqrt{2}-1)^{n}\right)^2\\=2^73^2ne_{\frac{n}{2}}^4f_{\frac{n}{2}}^2$.\\
\indent $Case$\, 2:  $n=2s$.\\
\indent  From (4.4), we know that the spanning tree number of
$C_4\times C_n$  of this case is
$\det(M_3)=2^8 3^2se_s^4f_s^2=2^73^2ne_{\frac{n}{2}}^4f_{\frac{n}{2}}^2$.
\end{proof}

\begin{cor}
For every $n\geq 3$, we have that
$\prod\limits_{j=1}^{n-1}\left(4-2\cos\frac{2\pi
j}{n}\right)^2\left(6-2\cos\frac{2\pi j}{n}\right)
=\frac{1}{4^{n+2}}\left((\sqrt{3}+1)^{n}-(\sqrt{3}-1)^{n}\right)^4\cdot
\left((\sqrt{2}+1)^{n}-(\sqrt{2}-1)^{n}\right)^2. $
\end{cor}

\begin{proof}
It is not difficult to know that the Laplacian eigenvalues of
$C_n$ are $(2-2\cos\frac{2\pi j}{n}),\ 0\leq j\leq n-1$. Then it
follows from the argument of the second section of [7] that the
Laplacian eigenvalues of $C_4\times C_n$ are: $0,\, 2,\, 2,\, 4,\,
2-2\cos\frac{2\pi j}{n}$, $4-2\cos\frac{2\pi j}{n}$\ (with
multiplicity 2), $6-2\cos\frac{2\pi j}{n}$, where $1\leq j\leq n-1$.
Then  by the well known Kirchhoff Matrix-Tree Theorem we know  the
spanning tree number of $C_4\times C_n$ is
$\frac{4}{n}\prod\limits_{j=1}^{n-1}\left(2-2\cos\frac{2\pi
j}{n}\right)\left(4-2\cos\frac{2\pi j}{n}\right)^2
\left(6-2\cos\frac{2\pi j}{n}\right).$ Since $C_n$ has $n$ spanning
trees, we have
$\frac{1}{n}\prod\limits_{j=1}^{n-1}(2-2\cos\frac{2\pi j}{n})=n.$
Thus the spanning tree number of $C_4\times C_n$  equals
$4n\prod\limits_{j=1}^{n-1}(4-2\cos\frac{2\pi
j}{n})^2(6-2\cos\frac{2\pi j}{n}).$\, Recall theorem 5.3, we have
that $4n\prod\limits_{j=1}^{n-1}\left(4-2\cos\frac{2\pi
j}{n}\right)^2\left(6-2\cos\frac{2\pi j}{n}\right)
=\frac{n}{4^{n+1}}\left((\sqrt{3}+1)^{n}-(\sqrt{3}-1)^{n}\right)^4\cdot
\left((\sqrt{2}+1)^{n}-(\sqrt{2}-1)^{n}\right)^2.$ So this corollary
holds.
\end{proof}

\end{document}